\documentclass[letterpaper, 10 pt, journal, twoside]{IEEEtran}  % Comment this line out
\IEEEoverridecommandlockouts                              % This command is only
\usepackage{amsmath, amssymb,amsthm} % assumes amsmath package installed
\usepackage{bm}
\usepackage{multirow}
\usepackage{makecell}

\usepackage[bb=dsserif]{mathalpha}

\newtheorem{theorem}{Theorem}[section]

\newtheorem{definition}[theorem]{Definition}

\newtheorem{lemma}[theorem]{Lemma}
\newtheorem{example}[theorem]{Example}
\newtheorem{remark}[theorem]{Remark}

\usepackage{graphicx} % support the \includegraphics command and options
\usepackage{epsfig} % for postscript graphics files% \usepackage{epstopdf}
\usepackage{epstopdf}
\usepackage{cite}
\usepackage[ruled,vlined]{algorithm2e}
\usepackage{bbm}
\usepackage{tikz}
\usetikzlibrary{shapes,arrows,positioning,calc,chains}

\makeatletter
\DeclareRobustCommand{\rvdots}{%
	\vbox{
		\baselineskip4\p@\lineskiplimit\z@
		\kern-\p@
		\hbox{.}\hbox{.}\hbox{.}
}}
\makeatother

\usepackage{lipsum}
\makeatletter
\def\footnoterule{\relax%
	\kern-5pt
	\hbox to \columnwidth{\hfill\vrule width .9\columnwidth height 0.4pt\hfill}
	\kern4.6pt}
\makeatother

\usepackage{color,hyperref}
\definecolor{darkblue}{rgb}{0.0,0.0,0.6}
\hypersetup{colorlinks,breaklinks,linkcolor=darkblue,urlcolor=darkblue,anchorcolor=darkblue,citecolor=darkblue}

\makeatother
\title{On Robustness of Double Linear Policy\\ with Time-Varying Weights}

\author{\large Xin-Yu Wang$^{*}$  and Chung-Han Hsieh${}^{**}$, \textit{Member, IEEE}% <-this % stops a space
	\thanks{This paper is partially supported by the Ministry of Science and Technology~(MOST), Taiwan, R.O.C. under Grant:  MOST-111-2221-E-007-124-.\hskip -10pt ${}^*$Xin-Yu Wang is a graduate student with the Department of Quantitative Finance, National Tsing Hua University, Hsinchu 300044, Taiwan R.O.C. E-mail: \href{mailto: xinyuwang@gapp.nthu.edu.tw}{xinyuwang@gapp.nthu.edu.tw}. } 
	\thanks{\hskip -10pt ${}^{**}$Chung-Han Hsieh is with the Department of Quantitative Finance, National Tsing Hua University, Hsinchu 300044, Taiwan R.O.C. E-mail: \href{mailto: ch.hsieh@mx.nthu.edu.tw}{ch.hsieh@mx.nthu.edu.tw}.}
}

\begin{document}

	\maketitle
	\thispagestyle{empty}
	\pagestyle{empty}
	
%	\parindent = 0pt
%%%%%%%%%%%%%%%%%%%%%%%%%%%%%%%%%%%%%%%%%%%%%%%%%%%%%%%%%%%%%%%%%%%%%%%%%%%%%%%%
	
\begin{abstract} 
In this paper, we extend the existing double linear policy by incorporating time-varying weights instead of constant weights and study a certain robustness property, called robust positive expectation (RPE),  in a discrete-time setting. 
 We prove that the RPE property holds by employing a novel elementary symmetric polynomials characterization approach and derive an explicit expression for both the expected cumulative gain-loss function and its variance.
To validate our theory, we perform extensive Monte Carlo simulations using various weighting functions. 
Furthermore, we demonstrate how this policy can be effectively incorporated with standard technical analysis techniques, using the moving average as a trading signal.
\end{abstract}
	
\vspace{3mm}
\begin{IEEEkeywords}
 Stochastic Systems, Finance, Robustness, Time-Varying Parameter Systems, Positive Systems. 
\end{IEEEkeywords}
	
%%%%%%%%%%%%%%%%%%%%%%%%%%%%%%%%%%%%%%%%%%%%%%%%%%%%%%%%%%%%%%%%%%%%%%%%%%%%%%%%

% \medskip
\section{Introduction}\label{section: introduction}
The Robust Positive Expectation (RPE) is a property that ensures a trading policy has a positive expected profit robustly, and it is closely related to the stochastic positivity of a dynamical system in the control area.
Some early work related to robustness issues in financial systems can be found in \cite{dokuchaev2002dynamic}.
Later, a strategy called Simultaneous Long-Short (SLS) was proposed; see~\cite{barmish2011arbitrage,barmish2015new}, and shown to guarantee the RPE in markets with asset prices governed by geometric Brownian motion (GBM).

Later, several extensions were proposed in the literature, including generalization for Merton’s diffusion model in~\cite{baumann2016stock}, GBM model with time-varying parameters in~\cite{primbs2017robustness}, and any linear stochastic differential equation (SDE) in~\cite{baumann2019positive}. 
Additionally, the SLS strategy was extended to the proportional-integral (PI) controller in \cite{malekpour2018generalization}, to the latency trading in \cite{malekpour2016stock},  
and coupled SLS strategy on pair trading for two correlated assets was studied in \cite{deshpande2018generalization, deshpande2020simultaneous}.  
In \cite{maroni2019robust}, a robust design strategy for
stock trading via feedback control is proposed.
\cite{o2020generalized} proposed a generalized SLS with different weight settings on long and short positions. Recently, \cite{hsieh2022generalization} considered a long-only affine feedback control with a stop-loss order.

In \cite{hsieh2022robust}, a modified SLS strategy,  called \textit{double linear policy}, was proposed to solve an optimal weight selection problem using the mean-variance approach in a discrete-time setting while preserving the RPE property.  Then \cite{hsieh2022robustness} established a sufficient condition of RPE when the transaction costs are present. 
However, previous work including \cite{hsieh2022robust, hsieh2022robustness} and many SLS literature assumed constant weight, investing the same proportion of account value in each stage.
This paper extends the  weight of double linear policy from constant to a broad class of time-varying functions in a discrete-time setting and proves that the RPE property still holds for this extension.

\subsection{Contributions of the Paper}
Proving an RPE property for a policy with time-varying weights is known to be challenging.\footnote{The conventional method for proving RPE of a trading policy with constant weight often relies on a key identity that $(1+x)^k + (1-x)^k>2$ for all~$k>1$ and $x \neq 0$. However, this approach may not apply when~$x$ varies over time, as in the case of the policies with time-varying weights.} This paper addresses this challenge by using a novel \textit{elementary symmetric polynomials} characterization approach.  We extend the existing results by showing that the RPE property holds for the double linear policy with time-varying weights. Closed-form expressions for the expected cumulative gain-loss function and its variance are provided. 
Additionally, we illustrate how the proposed policy can be incorporated with the common technical analysis technique.
The results presented in this paper contribute to the literature on robustness in financial systems.

\section{Problem Formulation}
For stage $k=0,1,2,\dotsc,$ let $S(k)>0$ be the \textit{underlying risky asset price} at stage $k$. Then the associated \textit{per-period return} is given by
$X(k):=\frac{S(k+1)-S(k)}{S(k)}.
$
Assume that~$X(k) \in [X_{\min},   X_{\max}]$ for all $k$ with probability one, and known bounds $-1 < X_{\min} < 0 < X_{\max} < \infty$. Additionally, assume that $X_{\min}$ and~$X_{\max}$ are in the support of $X(k)$.
Furthermore, assume that~$X(k)$ are independent with a common mean $\mathbb{E}[X(k)] = \mu \in \mathbb{R}$ and common variance~${\rm var}(X(k)) = \sigma^2 > 0$ for all~$k$.\footnote{This setting does not assume an underlying stochastic process governing the prices of the risky asset and is less restrictive than the typical independent and identically distributed returns assumption. }
In the sequel, we assume that the trades incur zero transaction costs and that the underlying asset has perfect liquidity. This setting serves as a good starting point for building the model and is closely related to the \textit{frictionless market} in finance; see \cite{merton1992continuous}.

\subsection{Double Linear Policy with Time-Varying Weights}
In~\cite{hsieh2022robust} and many SLS literature, the trading policy is proposed with constant weights. This paper extends the constant weights to a time-varying weighting function.
With initial account value~$V(0) := V_0 > 0$, we spilt it into two parts: Taking a fraction $\alpha \in [0,1]$, define $V_L(0) := \alpha V_0$ as the initial account value for \textit{long} position and $V_S(0) := (1-\alpha) V_0$ for \textit{short} position.
If $\alpha=1$, we are in a long-only position while~$\alpha =0$ corresponds to a pure short position.

The trading policy $\pi(\cdot)$ is given by~$\pi(k) := \pi_L(k)+\pi_S(k)$, where $\pi_L$ and $\pi_S$ are of double linear forms:
\begin{align} \label{eq: double linear policy}
    \begin{cases}
        \pi_L(k) = w(k) V_L(k);  \\
        \pi_S(k) = -w(k) V_S(k).
    \end{cases}
\end{align}
The weighting function $w(k) \in \mathcal{W} := [0, w_{\max}]$ for all~$k$ with $w_{\max} := \min\{1,1/X_{\max}\}$ and is assumed to be \textit{causal}; i.e., it may depend only on the information up to stage $k-1$. 
Any $w(k) \in \mathcal{W}$ is called \textit{admissible} weight. This condition is closely related to the survival trades; see Section~\ref{subsection: survivability considerations}. Hence, the account values under the double linear policy~$\pi_L$ and~$\pi_S$, denoted by $V_L(k)$ and $V_S(k)$, can be described as the following linear time-varying stochastic difference equation:
\begin{align*}
    \begin{cases}
        V_L(k+1) = V_L(k)+X(k)\pi_L(k) + (V_L(k) - \pi_L(k)) r_f;\\
        V_S(k+1) = V_S(k)+X(k)\pi_S(k),
    \end{cases}
\end{align*}
where $r_f \geq 0$ is a \textit{riskless} rate for a bank account or a treasury bond.\footnote{In practice, when shorting an asset, the corresponding proceeds are typically held as \textit{collateral} by the broker to cover any potential losses from the short position. These proceeds are generally not available for immediate reinvestment into a riskless asset, such as a bank account or treasury bond.}
Note that when $r_f >0$, account profit increases. Hence, as seen later in sections to follow, when studying the robustness of the double linear policy, we assume without loss of generality~$r_f :=0$. Then the account value for long position reduce to $V_L(k+1) = V_L(k) + X(k)\pi_L(k)$.
% This implies that
% $ 
% V_L(k)=\prod_{j=0}^{k-1} (1+w(j)X(j))V_L(0) 
% $
% and
% $
% V_S(k) = \prod_{j=0}^{k-1} (1-w(j)X(j))V_S(0).
% $
Therefore, the overall account value for both long and short positions at stage~$k$ is given by
 \begin{align*}
    V(k) 
    &=V_L(k)+V_S(k)
    =V_0 \left( \alpha R_+(k) + (1-\alpha)R_-(k) \right) ,
\end{align*}
where $R_+(k):=\prod_{j=0}^{k-1}(1+w(j)X(j) )$ and $R_-(k):=\prod_{j=0}^{k-1}(1-w(j)X(j))$.

\subsection{Survivability Considerations} \label{subsection: survivability considerations}
Fix $V_0 > 0$ and $\alpha \in (0,1)$, we ensure that the trades are survivable for all $k$; i.e., the $w$-value that can potentially lead to $V(k)<0$ is disallowed. To see this, for stage~$k=0,1,\dotsc $, fix~$w(k) \in \mathcal{W}$. We observe that for the long position, we have $ V_L(k) 
    % &= V_0 \alpha R_+(k) 
     \geq V_0 \alpha (1+ w_{\max} X_{\min})^k > 0
$
since $w_{\max} \leq 1$ and $X_{\min} >-1$.
On the other hand, for the short position, we also have
$    V_S(k) 
    % &= V_0 (1-\alpha) R_-(k) 
     \geq V_0 (1-\alpha) (1- w_{\max} X_{\max})^k \geq 0
$
since $w_{\max} \leq 1/X_{\max}$.
Therefore, the overall account value satisfies $V(k) = V_L(k) + V_S(k) >0$ for all $k$ with probability~one.
% \begin{align*}
%     V(k) &\geq 
%     V_0(\alpha (1+w_{\max}X_{\min})^k + (1-\alpha) (1-w_{\max}X_{\max})^k ).
% \end{align*}
% Since $w(k)\in \mathcal{W}$ for all $ k$ with $w_{\max}:=\min\{1,1/X_{\max}\}$ and $-1 < X_{\min} < 0 < X_{\max} < \infty$. We have $-1 < w_{\max}X_{\min} \leq w(j)X(j)\leq w_{\max}X_{\max} \leq 1$ for all $j$.
% It is readily verified that both of~$1+w_{\max}X_{\min} > 0$ and $1-w_{\max}X_{\max} \geq 0$. 
% Hence, $V(k) >0 $ for all $k$ with probability one.

\subsection{Robust Positive Expectation Problem}
% For $k > 0$, let $\mathcal{G}(k):= V(k) - V_0$ be the cumulative gain-loss function. 
% Then the expected cumulative gain-loss function denoted by 
% $\overline{\mathcal{G}}(k) := \mathbb{E}[\mathcal{G}(k)]$.
The primary objective of this paper is to study the following RPE problem.

\begin{definition}[Robust Positive Expectation] \rm
    For stage $k = 0, 1, \dots$, let $V_0>0$ be the initial account value, and $V(k)$ be the account value at stage $k$. Define the expected cumulative gain-loss function up to stage $k$ as  $\overline{\mathcal{G}}(k):= \mathbb{E}[V(k)] - V_0$. A trading policy is said to have a \textit{robust positive expectation}~(RPE) property if it ensures that $\overline{\mathcal{G}}(k)>0$ for all $k>1$ and under all market conditions.
\end{definition}

\section{Gain-Loss Analysis}
For $k > 0$, let $X:=\{X(j)\}_{j=0}^{k-1}$ and ${\bm w}:=\{w(j)\}_{j =0}^{k-1}$.
With $V_0>0$, consider the double linear policy with $\alpha \in (0,1)$ and weight~$w(k) \in \mathcal{W}$ for all $k$.
The \textit{cumulative trading gain-loss function} up to stage $k$ is given by
\begin{align*}
    \mathcal{G}(\alpha, \bm{w}, k, X) &:= V(k) - V_0 \\
    &=  V_0 (\alpha R_+(k) + (1-\alpha)R_-(k)-1),
\end{align*}
and the expectation is $\overline{\mathcal{G}}(\alpha, {\bm w}, k, \mu):= \mathbb{E}[\mathcal{G}(\alpha, \bm{w}, k, X)]$.
If the weights are constant; i.e.,~$w(k):=w$ for all $k$, then the RPE property is readily established when $\alpha = 1/2$, see~\cite{hsieh2022robustness}. However, difficulties arise when the weighting function is time-varying. 
  To address this, a set of \textit{elementary symmetric polynomials}\footnote{ We say that $e(\cdot)$ is a \textit{symmetric polynomial} if for any permutation~$\sigma$ of the subscripts $1, 2, \cdots, n$, it follows that $e(x_{\sigma(1)}, x_{\sigma(2)}, \cdots, x_{\sigma(n)}) = e(x_1, x_2, \cdots, x_n)$.} in $k$ variables, $\{w(0), \dots, w(k-1)\}$, are considered and defined as $\{e_1(k), e_2(k),\dots e_k(k)\}$ with
\[
e_j(k) := \sum_{0 \leq i_1<i_2< \dotsb < i_j \leq k-1} w(i_1)w(i_2) \dotsm w(i_j)  
% \sum_{\{(i_1,\dotsc , i_j):0 \leq i_1<i_2< \dotsb < i_j \leq k-1\}}\Big(w(i_1)w(i_2) \dotsm w(i_j)\Big)
\] 
for $i_j \in \mathbb{N}$.
Note that $e_j(k)\geq 0$ for all~$j$ and~$k$, 
% With $w:=\{w(i_j)\}_{j=0}^{k-1}$, we define an auxiliary function
% \[
% \mathcal{M}(k,j,w,\mu ):= e_j(k) \mu^j,
% % \sum_{\{(i_1,\dotsc , i_j):0 \leq i_1< \dotsb < i_j \leq k-1\}}\Big(w(i_1)w(i_2) \dotsm w(i_j)\Big) \mu^j,
% \]
which is the sum of the $j$th multiplication term of admissible weights.
The following example illustrates the calculation of elementary symmetric polynomials.
% and the auxiliary functions.
% Then the following lemma is useful to prove the RPE property.

\begin{example}[Elementary Symmetric Polynomials] \rm
This example illustrates the calculation of the elementary symmetric polynomials $e_j(k)$.
% and auxiliary function $\mathcal{M}(k, j, w, \mu)$.
Specifically, for $k=1$, the polynomials to be calculate is $\{e_1(1)\}$ which is given by
$
e_1(1) = w(0).
$
For stage~$k=2$, the elementary symmetric polynomials~$\{e_1(2), e_2(2)\}$ are given~by
\begin{align*}
    e_1(2) &= \sum_{0\leq i_1 \leq 1} w(i_1) = w(0) + w(1); \\
    e_2(2) &= \sum_{0 \leq i_1 < i_2 \leq 1} w(i_1)w(i_2) = w(0)w(1).
\end{align*}
Similarly, for $k=3$, the elementary symmetric polynomials~$\{e_1(3), e_2(3), e_3(3)\}$  becomes
\begin{align*}
    e_1(3) &= \sum_{0\leq i_1 \leq 2} w(i_1) = w(0) + w(1) + w(2); \\
    e_2(3) &= \sum_{0 \leq i_1 < i_2 \leq 2} w(i_1)w(i_2) \\
    &= w(0)w(1) + w(0)w(2) + w(1)w(2); \\
    e_3(3) &= \sum_{0 \leq i_1 < i_2 < i_3 \leq 2} w(i_1)w(i_2)w(i_3) = w(0)w(1)w(2). 
\end{align*}

\end{example}

% \begin{remark} \rm
    As seen later in this section, the representation of elementary symmetric polynomials is useful for proving the RPE property; see Lemmas~\ref{lemma: representation of long-short returns} and \ref{lemma: positive polynomial} to follow.
% \end{remark}
Define shorthand notations $\overline{R}_+(k) := \mathbb{E}[R_+(k)]$
% = $\prod_{j=0}^{k-1}(1+w(j)\mu)$
and $\overline{R}_-(k) := \mathbb{E}[R_-(k)]$. 
With the aid of the independence of $X(k)$, it follows that $\overline{R}_+(k) = \prod_{j=0}^{k-1}(1+w(j)\mu)$ and $\overline{R}_-(k) = \prod_{j=0}^{k-1}(1-w(j)\mu)$.

\begin{lemma} \label{lemma: representation of long-short returns}
Fix $m \geq 1.$ Let $\alpha \in (0,1)$ and $w(k) \in \mathcal{W}$ for all~$k$, $\overline{R}_+(k)$ and $\overline{R}_-(k)$ for stage $k = 2m+1$ satisfies 
    \begin{align*}
        % \mathbb{E}[R_+(k)] 
        \overline{R}_+(k)& = 1 + \sum_{j=0}^{m} e_{2j+1}(k) \mu^{2j+1} + \sum_{j=1}^{m} e_{2j}(k) \mu^{2j}; \\
        % \mathbb{E}[R_-(k)] 
        \overline{R}_-(k)& = 1 - \sum_{j=0}^{m}e_{2j+1}(k)\mu^{2j+1} + \sum_{j=1}^{m} e_{2j}(k)\mu^{2j}.
    \end{align*}
   On the other hand, for $k =2m$, $\overline{R}_+(k)$ and~$\overline{R}_-(k)$  satisfies
    \begin{align*}
        % \mathbb{E}[R_+(k)] 
        \overline{R}_+(k) & = 1+ \sum_{j=0}^{m-1}e_{2j+1}(k)\mu^{2j+1} + \sum_{j=1}^{m}e_{2j}(k)\mu^{2j}; \\
        % \mathbb{E}[R_-(k)] 
        \overline{R}_-(k)& = 1- \sum_{j=0}^{m-1}e_{2j+1}(k)\mu^{2j+1} + \sum_{j=1}^{m} e_{2j}(k)\mu^{2j}.
    \end{align*}
\end{lemma}

\begin{proof}
    We use a shorthand notation $w_j$ for $w(j)$ in the proof.
   Fix $m \geq 1$. 
   % We recall that
   % $R_+(k):=\prod_{j=0}^{k-1} (1+w(j)X(j))$ and $R_-(k):=\prod_{j=0}^{k-1}(1-w(j)X(j))$.
   Now for the case $k = 2m+1$, which is an odd number, $\overline{R}_+(k)$ is given by
    \begin{align*}
        % \mathbb{E}[R_+(k)]
        \overline{R}_+(k)
        % & = \mathbb{E}\left[\prod_{j=0}^{k-1} (1+w_j X(j))\right]\\ 
        &= \prod_{j=0}^{k-1} (1+w_j\mu) \\
        & = (1+w_0\mu)(1+w_1\mu) \dotsm (1+w_{k-1}\mu)\\
        & = 1 + e_1(k)\mu + \dotsb + e_k(k)\mu^k\\
        & = 1 + \sum_{j=1}^{k} e_j(k) \mu^j \\
        & = 1 + \sum_{j=0}^{m} e_{2j+1}(k) \mu^{2j+1} + \sum_{j=1}^{m} e_{2j}(k) \mu^{2j},
        % &\quad = 1 + \sum_{j=0}^{m} \mathcal{M}(k,2j+1,w,\mu) + \sum_{j=1}^{m} \mathcal{M}(k,2j,w,\mu).
    \end{align*}
where the last equality separates terms into odd and even cases.    
Likewise, $\overline{R}_-(k)$ for $k=2m+1$ is 
    \begin{align*}
        % &\mathbb{E}[R_-(k)]
        &\overline{R}_-(k)
        =\prod_{j=0}^{k-1} (1-w_j\mu) \\
        &\quad = 1 + \sum_{j=1}^{k} (-1)^j e_j(k) \mu^j \\
        &\quad = 1 + \sum_{j=0}^{m}(-1)^{2j+1}e_{2j+1}(k) \mu^{2j+1} + \sum_{j=1}^{m}(-1)^{2j}e_{2j}(k) \mu^{2j}\\
        &\quad = 1 - \sum_{j=0}^{m}e_{2j+1}(k)\mu^{2j+1} + \sum_{j=1}^{m} e_{2j}(k)\mu^{2j}.
        % &\quad = 1+ \sum_{j=0}^{m}(-1)^{2j+1}\mathcal{M}(k,2j+1,w,\mu) \\
        % &\quad \quad + \sum_{j=1}^{m} (-1)^{2j} \mathcal{M}(k,2j,w,\mu) \\
        % &\quad = 1- \sum_{j=0}^{m}\mathcal{M}(k,2j+1,w,\mu) + \sum_{j=1}^{m}\mathcal{M}(k,2j,w,\mu).
    \end{align*}
On the other hand, for the even number case $k = 2m$,  with an almost identical argument, it is readily verified that
    $
        \overline{R}_+(k)
        % \mathbb{E}[R_+(k)]
         % &= \prod_{j=0}^{k-1} (1+w_j\mu) \\
        % &= 1 + \sum_{j=1}^{k} e_j(k) \mu^j \\
        % &= 1 + \sum_{j=0}^{m-1} e_{2j+1}(k) \mu^{2j+1} + \sum_{j=1}^{m} e_{2j}(k) \mu^{2j} \\
        % &= 1 + \sum_{j=0}^{m-1} \mathcal{M}(k,2j+1,w,\mu) + \sum_{j=1}^{m} \mathcal{M}(k,2j,w,\mu) \\
        = 1+ \sum_{j=0}^{m-1}e_{2j+1}(k)\mu^{2j+1} + \sum_{j=1}^{m}e_{2j}(k)\mu^{2j}
    $
and
    $
     \overline{R}_-(k)
     % \mathbb{E}[R_-(k)] 
      % &=\prod_{j=0}^{k-1} (1-w_j\mu) \\
        % &= 1 + \sum_{j=1}^{k} (-1)^j e_j(k) \mu^j \\ 
        % &= 1+ \sum_{j=0}^{m-1} (-1)^{2j+1}\mathcal{M}(k,2j+1,w,\mu) \\
        % &\quad + \sum_{j=1}^{m} (-1)^{2j} \mathcal{M}(k,2j,w,\mu) \\
        % &= 1- \sum_{j=0}^{m-1}\mathcal{M}(k,2j+1,w,\mu) + \sum_{j=1}^{m}\mathcal{M}(k,2j,w,\mu)\\
        = 1- \sum_{j=0}^{m-1}e_{2j+1}(k)\mu^{2j+1} + \sum_{j=1}^{m} e_{2j}(k)\mu^{2j}
    $
    and the proof is complete.
\end{proof}

\begin{lemma} \label{lemma: positive polynomial}
    For $k  > 1$,   $e_{2}(k)>0$ provided that at least two weights $w(i), w(j) >0$ for some $i,j \in \{0,1,\dots, k-1\}$ and~$i \neq j$.
\end{lemma}

\begin{proof}
Fix $k > 1$. Then
$	
e_{2}(k) := \sum_{0 \leq i_1<i_2 \leq k-1} w(i_1)w(i_2)  .
$
Proceed a proof by induction. If $k=2$, which corresponds to  
$
     e_{2}(2) = w(0)w(1) .
$
Since we are assuming that at least two weights are strictly positive, in this case, it corresponds to $w(0), w(1) >0 $.  Therefore,  $e_2(2) >0.$ 
Next, assuming that~$e_{2}(k) >0$ for at least two  weights $w(i), w(j) >0$ for some $i,j \in \{0,1,\dots, k-1\}$, we must show $e_{2}(k+1) >0.$
Note that
\begin{align*}
e_{2}(k+1) 
    & = \sum_{0 \leq i_1 < i_2 \leq k} w(i_1)w(i_2) \\
    &= \sum_{0 \leq i_1<i_2 \leq k-1} w(i_1)w(i_2)  + \sum_{i=1}^{k} w(k-i)w(k) \\
    &=  e_2(k) +  \sum_{i=1}^{k} w(k-i)w(k)  >0,
\end{align*}
where the last inequality holds by inductive hypothesis that~$e_2(k) >0$ for at least two weights, say $w(i), w(j) >0$ for some $i,j \in \{0,1,\dots, k-1\}$  and the fact that the sum~$\sum_{i=1}^{k} w(k-i)w(k)  \geq 0$.
\end{proof}

\begin{theorem}[RPE with Time-Varying Weights] \label{theorem: RPE with time-varying weights}
  Let $V_0>0$.  Consider a double linear policy with $\alpha \in (0,1)$ and weights~$w(k)\in \mathcal{W}$ for all $k$. Then, the \textit{expected cumulative gain-loss function} is given by
    \begin{align*}
        \overline{\mathcal{G}}(\alpha, {\bm w}, k, \mu) 
        &= V_0\left( \alpha \overline{R}_+(k)+(1-\alpha)\overline{R}_-(k) -1 \right).
    \end{align*}
    Moreover, when $\alpha = 1/2$ and $w(k) \in \mathcal{W}$ with at least two weights being strictly positive, the RPE property holds; i.e.,~$
    \overline{\mathcal{G}}(\alpha, {\bm w}, k, \mu)>0
    $ for $k>1$ and all \(\mu \neq 0\). 
\end{theorem}

\begin{proof}
    %%%%%%%%%%%%%% cumulative gain-loss closed form %%%%%%%%%%%%%%%%%%%%%%%%%%%%%
    To calculate the expected cumulative gain-loss function, we use the fact that per-period returns $X(k)$ are independent with common mean $\mathbb{E}[X(k)] = \mu$ for all $k$. Thus, it is readily verified that
    \begin{align*}
        \overline{\mathcal{G}}(\alpha, \bm{w}, k, \mu)  
        &= \mathbb{E}\left[ V_0 \left( \alpha R_+(k) + (1-\alpha) R_-(k) -1 \right) \right] \\
        % &= \mathbb{E}\Bigg[ V_0 \bigg( \alpha \prod_{j=0}^{k-1} (1+w(j)X(j)) + (1-\alpha) \prod_{j=0}^{k-1}(1-w(j)X(j)) -1 \bigg) \Bigg] \\
        &= V_0 \left( \alpha \mathbb{E}[R_+(k)] + (1-\alpha) \mathbb{E}[R_-(k)] -1 \right) \\
        &= V_0 \left( \alpha \overline{R}_+(k) + (1-\alpha) \overline{R}_-(k)-1 \right),
        % &= V_0 \left( \alpha \prod_{j=0}^{k-1}(1+w(j) \mu) + (1-\alpha) \prod_{j=0}^{k-1}(1-w(j)\mu)-1 \right),
    \end{align*}
    which is identical to the desired equality in the statement of the theorem.
    To complete the proof, we now show that the RPE property holds. Fix $k>1$. Consider two cases by splitting~$k$ into odd and even numbers. 
    We begin by considering~$k=2m+1$ with $m\geq 1$, corresponding to an odd number. Then, using Lemma~\ref{lemma: representation of long-short returns} for the odd case, we have 
    \begin{align*}
        &\overline{\mathcal{G}}(\alpha, \bm{w}, k, \mu) \\
        % &= V_0 \bigg( \alpha \mathbb{E}[R_+(k)] + (1-\alpha) \mathbb{E}[R_-(k)] \bigg) \\
        % &= V_0 \bigg( \alpha \prod_{j=0}^{k-1}(1+w(j) \mu) + (1-\alpha) \prod_{j=0}^{k-1}(1-w(j)\mu)-1 \bigg) \\
        &= V_0 \bigg( \alpha \Big(1 + \sum_{j=0}^{m}e_{2j+1}(k)\mu^{2j+1} + \sum_{j=1}^{m}e_{2j}(k)\mu^{2j} \Big) \\
        &\quad + (1-\alpha)\Big(1 - \sum_{j=0}^{m}e_{2j+1}(k)\mu^{2j+1} + \sum_{j=1}^{m}e_{2j}(k)\mu^{2j}\Big) -1 \bigg)\\
        &= V_0\bigg( (2\alpha-1) \sum_{j=0}^{m}e_{2j+1}(k) \mu^{2j+1} + \sum_{j=1}^{m}e_{2j}(k) \mu^{2j} \bigg) \\
        &= V_0\bigg( (2\alpha-1)\mu \sum_{j=0}^{m}e_{2j+1}(k) \mu^{2j} + \sum_{j=1}^{m}e_{2j}(k) \mu^{2j} \bigg).
        % &= V_0 \Bigg( \alpha \bigg( 1+ \sum_{j=0}^{m} \mathcal{M}(k, 2j+1) + \sum_{j=1}^{m} \mathcal{M}(k,2j) \bigg) + \\
        % &\qquad\quad (1-\alpha) \bigg( 1-\sum_{j=0}^{m} \mathcal{M}(k,2j+1) + \sum_{j=1}^{m}\mathcal{M}(k,2j) \bigg) -1 \Bigg)\\
        % &= V_0 \Bigg( \Big( \alpha+(1-\alpha) -1 \Big) + \Big(\alpha -(1-\alpha)\Big) \sum_{j=0}^{m}\mathcal{M}(k,2j+1) + \Big(\alpha+ (1-\alpha)\Big) \sum_{j=1}^{m} \mathcal{M}(k,2j) \Bigg)\\
        % &= V_0 \Bigg( \Big(2\alpha -1\Big) \sum_{j=0}^{m} \mathcal{M}(k,2j+1) + \sum_{j=1}^{m}\mathcal{M}(k,2j) \Bigg)\\
        % &= V_0 \Bigg( \Big(2\alpha -1\Big) \sum_{j=0}^{m} \mathcal{M}_O(k,2j+1)\mu + \sum_{j=1}^{m}\mathcal{M}(k,2j) \Bigg)
    \end{align*}
    % where 
        % &\mathcal{M}(k,j,w,\mu ):= e_j(k)\mu^j \\
        % &e_j(k) := \sum_{0 \leq i_1<i_2< \dotsb < i_j \leq k-1}\Big(w(i_1)w(i_2) \dotsm w(i_j)\Big) \\
        % &\mu \mathcal{M}_O(k,j,w,\mu ):=\mathcal{M}(k,j,w,\mu ), \\
        % &\text{i.e. }
        % \mathcal{M}_O(k,j,w,\mu )= e_j(k)\mu^{(j-1)}.
Since $w(k) \in \mathcal{W}$, we have 
$e_j(k) \geq 0$ for all $j$ and $k$. 
Hence, it follows that~$e_{2j+1}(k) \mu^{2j} \geq 0$ and $e_{2j}(k) \mu^{2j} \geq 0$.  
In addition, for~$\alpha = {1}/{2}$, the expected cumulative gain-loss function becomes 
\begin{align*}
\overline{\mathcal{G}}(\alpha, {\bm w}, k, \mu) 
	&= V_0 \sum_{j=1}^{m} e_{2j}(k) \mu^{2j} \\
	&= V_0 \left(  e_{2}(k) \mu^{2}  + \sum_{j=2}^{m} e_{2j}(k) \mu^{2j} \right).
\end{align*}
Since $V_0 > 0$, $\mu \neq 0$ and at least two weights are strictly positive $w(i), w(j) > 0$ for some $i, j$ with $i \neq j$, Lemma~\ref{lemma: positive polynomial} indicates that $e_2(k) >0$. It follows that $\overline{\mathcal{G}}(\alpha, \bm{w}, k, \mu)  >0$.
% That is the RPE property holds when $\mu \neq 0$ and at least two stage weights are non-zero.
% have $\overline{\mathcal{G}}(\alpha, w, k, \mu) >0$ for $\mu \neq 0$; i.e., an RPE property is seen.
    On the other hand, consider the case $k=2m$, which is an even number. 
    Using the second part of Lemma~\ref{lemma: representation of long-short returns}, we obtain 
    {\small 
    \begin{align*}
        \overline{\mathcal{G}}(\alpha, \bm{w}, k, \mu)
        &= V_0 \bigg( (2\alpha-1)\mu \sum_{j=0}^{m-1} e_{2j+1}(k) \mu^{2j} + \sum_{j=1}^{m}e_{2j}(k)\mu^{2j} \bigg). 
        % &= V_0 \Bigg( \alpha \bigg( 1+ \sum_{j=0}^{m-1} \mathcal{M}(k, 2j+1,w,\mu) + \sum_{j=1}^{m} \mathcal{M}(k,2j,w,\mu) \bigg) \\
        % &\quad + (1-\alpha) \bigg( 1-\sum_{j=0}^{m-1} \mathcal{M}(k,2j+1,w,\mu) \\
        % &\quad + \sum_{j=1}^{m}\mathcal{M}(k,2j,w,\mu) \bigg) -1 \Bigg)\\
        % &= V_0 \Bigg( (2\alpha -1) \sum_{j=0}^{m-1} \mu \mathcal{M}_O(k,2j+1,w,\mu) \\
        % &\qquad+ \sum_{j=1}^{m}\mathcal{M}(k,2j,w,\mu) \Bigg) 
    \end{align*}
    }A similar argument can be made for showing that~$e_{2j}(k) \mu^{2j} \geq 0$ for all $j$ and $k$.
    % $\mathcal{M}_O(k,2j+1,w,\mu)  \geq 0$ and  $\mathcal{M}(k,2j,w,\mu) \geq 0$. 
    Hence, taking $\alpha = 1/2$, using the fact that at least two weights are strictly positive, and Lemma~\ref{lemma: positive polynomial}, we again have $\overline{\mathcal{G}}(\alpha, \bm{w}, k, \mu)  > 0$ when $\mu \neq 0$, which completes the~proof.
\end{proof}

%%%%%%%%%%%%%%%%%%%%%%%%%%%%%%%%%%%%%%%%%%

\begin{remark} \rm \label{remark: RPE remark}
$(i)$. Theorem~\ref{theorem: RPE with time-varying weights} can be viewed as an extension of the existing RPE result using double linear policy with constant weights stated in \cite{hsieh2022robust}. That is, by taking $w(k) := w$ for all $k$, one readily obtains
\begin{align*}
    \overline{\mathcal{G}}(\alpha, \bm{w}, k, \mu)  
        % &= V_0\left( \alpha\prod_{j=0}^{k-1}(1+w\mu)+(1-\alpha)\prod_{j=0}^{k-1}(1-w\mu) -1 \right)\\
        &= V_0\left( \alpha (1+w\mu)^k +(1-\alpha) (1-w\mu)^k -1 \right).
\end{align*}
If $\alpha = 1/2$ and $w \in \mathcal{W}\setminus \{0\}$,  the desired strict positivity holds; i.e., $\overline{\mathcal{G}}(\alpha, \bm{w}, k, \mu) > 0 $ for $\mu \neq 0$ and all $k$.
$(ii)$. According to Theorem~\ref{theorem: RPE with time-varying weights}, it is readily verified that the expected cumulative gain-loss function  satisfies
    $ \overline{\mathcal{G}}(\alpha, \bm{w}, k, \mu) > 0$ for all $k $
     if~${\rm sgn}( (2\alpha - 1) \mu )>0$.
     % for~\(\mu > 0\), if \(\alpha > 1/2\) we have \(\overline{\mathcal{G}}(\alpha, \bm{w}, k, \mu)>0\) for all $k$. On the other hand,  for~\(\mu < 0\) if \(\alpha < 1/2\) we have 
%     $
%     \overline{\mathcal{G}}(\alpha, \bm{w}, k, \mu)>0,
%     $ for all $k$.
\end{remark}

\begin{lemma}[Variance of the Gain-Loss Function] \label{lemma: variance of the gain-loss function}
Let $V_0>0$. Consider a double linear policy with $\alpha \in (0,1)$ and weights $w(k) \in \mathcal{W}$ for all $k$ then 
the variance of the cumulative gain-loss function is given by  
\begin{align*}
    &{\rm var}(\mathcal{G}(\alpha , \bm{w}, k, X)) \\
    % &= \alpha^2 V_0^2 \prod_{j=0}^{k-1} \bigg( w(j)^2 \sigma^2 + \Big(1+w(j)\mu\Big)^2 \bigg) \\
    % &\quad+ (1-\alpha)^2 V_0^2 \prod_{j=0}^{k-1} \bigg( w(j)^2 \sigma^2 + \Big( 1-w(j)\mu \Big)^2 \bigg)\\
    % &\quad +2\alpha (1-\alpha)V_0^2 \prod_{j=0}^{k-1} \bigg( 1-w(j)^2 (\sigma^2 + \mu^2) \bigg) \\
    % &\quad- 2\alpha V_0^2 \prod_{j=0}^{k-1} \Big( 1+w(j)\mu\Big) \\
    % &\quad - 2(1-\alpha) V_0^2 \prod_{j=0}^{k-1} \Big( 1-w(j)\mu \Big) + V_0^2 \\
    % &\quad - V_0^2\bigg( \alpha \prod_{j=0}^{k-1} ( 1+w(j)\mu ) + (1-\alpha) \prod_{j=0}^{k-1} ( 1-w(j)\mu ) -1 \bigg)^2.
    &= V_0^2 \Bigg( \alpha^2 \prod_{j=0}^{k-1} \Big(w(j)^2 \sigma^2 + (1+w(j)\mu)^2  \Big) \\
        &\qquad + (1-\alpha)^2 \prod_{j=0}^{k-1} \Big(w(j)^2\sigma^2 + (1-w(j)\mu)^2 \Big) \\
        &\qquad + 2 \alpha (1-\alpha) \prod_{j=0}^{k-1}\Big(1-w(j)^2(\sigma^2 + \mu^2)\Big) \\
        &\qquad - 2\alpha(1-\alpha)\prod_{j=0}^{k-1}(1-w(j)^2\mu^2) \\
        &\qquad -\alpha^2 \prod_{j=0}^{k-1} (1+w(j) \mu)^2 - (1-\alpha)^2 \prod_{j=0}^{k-1} (1-w(j)\mu)^2 \Bigg). 
\end{align*}
\end{lemma}

\begin{proof}
The proof is based on straightforward calculation on ${\rm var}(\mathcal{G}(\alpha, \bm{w}, k, X))
        = \mathbb{E}[\mathcal{G}^2(\alpha, \bm{w}, k, X)] - \overline{\mathcal{G}}^2(\alpha, \bm{w}, k, \mu)$.
We first calculate the second moment of the gain-loss function:
% {\small 
% \begin{align*}
%     \mathbb{E}[\mathcal{G}^2(\alpha, w, k, X)] 
%     % & = \mathbb{E} \Bigg[ V_0^2 \bigg(\alpha \prod_{j=0}^{k-1}(1+w(j)X(j)) \\
%     % &\qquad  +(1-\alpha)\prod_{j=0}^{k-1}(1-w(j)X(j)) - 1 \bigg)^2 \Bigg] \\
%     &= V_0^2\Bigg( \alpha^2 \mathbb{E}\bigg[\prod_{j=0}^{k-1}(1+w(j)X(j))^2\bigg] \\
%     &\qquad \quad+ (1-\alpha)^2 \mathbb{E}\Bigg[ \prod_{j=0}^{k-1} (1-w(j)X(j))^2 \Bigg] + 1 \\
%     &\qquad\quad+ 2 \alpha (1-\alpha)\mathbb{E}\bigg[\prod_{j=0}^{k-1} (1-w(j)^2X(j)^2)\bigg] \\
%     &\qquad\quad - 2\alpha \mathbb{E}\bigg[\prod_{j=0}^{k-1} (1+w(j)X(j))\bigg] \\
%     &\qquad\quad - 2 (1-\alpha) \mathbb{E}\bigg[\prod_{j=0}^{k-1} (1-w(j)X(j)) \bigg] \Bigg).
% \end{align*}
% }
With the aid of the independence of $X(k)$, a lengthy but straightforward calculation leads to
{\small 
\begin{align} \label{eq: E[g_square]}
    &\mathbb{E}[\mathcal{G}^2(\alpha, \bm{w}, k, X)] \nonumber \\
    &= V_0^2 \Bigg(\alpha^2 \prod_{j=0}^{k-1} \Big(w(j)^2 \sigma^2 + (1+w(j)\mu)^2\Big) \nonumber \\
    &\qquad + (1-\alpha)^2 \prod_{j=0}^{k-1} \Big(w(j)^2\sigma^2 + (1-w(j)\mu)^2  \Big) + 1 \nonumber \\
    &\qquad + 2 \alpha (1-\alpha) \prod_{j=0}^{k-1}\Big(1-w(j)^2(\sigma^2 + \mu^2)\Big) \nonumber \\
    &\qquad - 2 \alpha \prod_{j=0}^{k-1} ( 1+w(j)\mu )  - 2 (1-\alpha)\prod_{j=0}^{k-1}( 1-w(j)\mu ) \Bigg).
\end{align}
}Then we calculate the square of the expected cumulative gain-loss function. That is,
{\small 
\begin{align} \label{eq: g_bar_square}          
&\overline{\mathcal{G}}^2(\alpha, \bm{w}, k, \mu) \nonumber \\
    &= V_0^2 \bigg( \alpha \overline{R}_+(k) + (1-\alpha) \overline{R}_-(k) -1 \bigg)^2 \nonumber \\
    &= V_0^2 \bigg( \alpha^2 \prod_{j=0}^{k-1} (1+w(j) \mu)^2 + (1-\alpha)^2 \prod_{j=0}^{k-1} (1-w(j)\mu)^2 \nonumber \\
    &\qquad + 1 +2\alpha(1-\alpha)\prod_{j=0}^{k-1}(1-w(j)^2\mu^2) \nonumber \\
    &\qquad - 2\alpha \prod_{j=0}^{k-1}(1+w(j)\mu) - 2(1-\alpha)\prod_{j=0}^{k-1}(1-w(j)\mu)\bigg).
    % &\qquad+2\alpha(1-\alpha)\prod_{j=0}^{k-1}(1-w(j)^2\mu^2) \bigg)
\end{align}
}In combination with Equations~(\ref{eq: E[g_square]}) and (\ref{eq: g_bar_square}), a lengthy but straightforward calculation again leads to the desired expression for the variance of the gain-loss function.
    % A lengthy but straightforward calculation leads to the variance of the gain-loss function:
    % \begin{align*}
    %     &{\rm var}(\mathcal{G}(\alpha , w, k, X))\\
    %     &= \mathbb{E}[\mathcal{G}^2(\alpha, w, k, X)] - \overline{\mathcal{G}}^2(\alpha, w, k, \mu) \\
    %     % &= V_0^2 \Bigg( \alpha^2 \prod_{j=0}^{k-1} \Big(1+2w(j)\mu+w(j)^2 (\sigma^2 + \mu^2)\Big) \\
    %     % &\quad + (1-\alpha)^2 \prod_{j=0}^{k-1} \Big(1-2w(j)\mu + w(j)^2(\sigma^2 + \mu^2) \Big) + 1 \\
    %     % &\quad + 2 \alpha (1-\alpha) \prod_{j=0}^{k-1}\Big(1-w(j)(\sigma^2 + \mu^2)\Big) \\
    %     % &\quad - 2 \alpha \prod_{j=0}^{k-1}( 1+w(j)\mu )  - 2(1-\alpha) \prod_{j=0}^{k-1}( 1-w(j)\mu ) \Bigg) \\
    %     % % &\quad - V_0^2 \bigg( \alpha \prod_{j=0}^{k-1} (1+w(j)\mu) + (1-\alpha) \prod_{j=0}^{k-1}(1-w(j)\mu) -1 \bigg)^2
    %     % &\quad - V_0^2 \Bigg( \alpha^2 \prod_{j=0}^{k-1} (1+w(j) \mu)^2 + (1-\alpha)^2 \prod_{j=0}^{k-1} (1-w(j)\mu)^2  \\
    %     % &\quad + 1 - 2\alpha \prod_{j=0}^{k-1}(1+w(j)\mu) - 2(1-\alpha)\prod_{j=0}^{k-1}(1-w(j)\mu)\\
    %     % &\quad+2\alpha(1-\alpha)\prod_{j=0}^{k-1}(1-w(j)^2\mu^2) \Bigg) \\
    %     &= V_0^2 \Bigg( \alpha^2 \prod_{j=0}^{k-1} \Big(w(j)^2 \sigma^2 + (1+w(j)\mu)^2  \Big) \\
    %     &\qquad + (1-\alpha)^2 \prod_{j=0}^{k-1} \Big(w(j)^2\sigma^2 + (1-w(j)\mu)^2 \Big) \\
    %     &\qquad + 2 \alpha (1-\alpha) \prod_{j=0}^{k-1}\Big(1-w(j)^2(\sigma^2 + \mu^2)\Big)  \\
    %     & \qquad - 2\alpha(1-\alpha)\prod_{j=0}^{k-1}(1-w(j)^2\mu^2) \\
    %     &\qquad -\alpha^2 \prod_{j=0}^{k-1} (1+w(j) \mu)^2 - (1-\alpha)^2 \prod_{j=0}^{k-1} (1-w(j)\mu)^2 \Bigg)
    % \end{align*}
    % which is desired.
\end{proof}

\begin{remark} \rm
  If $w(k):=w$ for all $k$, Lemma~\ref{lemma: variance of the gain-loss function} reduces to the variance expression obtained in \cite[Lemma 3.1]{hsieh2022robust}.
%     $(ii)$ If $\alpha = 1/2$, i.e., the case where RPE holds, then the corresponding variance is
% \begin{align*}
%     &{\rm var}(\mathcal{G}(\alpha , w, k, X)) \\
%     &= V_0^2 \Bigg( \frac{1}{4} \prod_{j=0}^{k-1} \Big(w(j)^2 \sigma^2 + (1+w(j)\mu)^2  \Big) \\
%         &\qquad + \frac{1}{4} \prod_{j=0}^{k-1} \Big(w(j)^2\sigma^2 + (1-w(j)\mu)^2 \Big)  \\
%         &\qquad + \frac{1}{2} \prod_{j=0}^{k-1}\Big(1-w(j)^2(\sigma^2 + \mu^2)\Big) - \frac{1}{2}\prod_{j=0}^{k-1}(1-w(j)^2\mu^2) \\
%         &\qquad -\frac{1}{4} \prod_{j=0}^{k-1} (1+w(j) \mu)^2 - \frac{1}{4} \prod_{j=0}^{k-1} (1-w(j)\mu)^2 \Bigg).
% \end{align*}
% \textbf{\color{red} What is the Upper bound for the variance? Is there any alpha that yields a lower var than alpha  = 1/2?}
\end{remark}
 
\section{Illustrative Examples}
This section illustrates the robustness of the double linear policy with time-varying weights using various examples. 
 
 \begin{example}[GBM with Jumps] \rm
 \label{example: GBM with jumps}
We now collect historical daily prices for Apple Inc. (Ticker:~AAPL) over a one-year period from January~2022 to December~2022.\footnote{Note that this one-year period provides a good test case since 2022 is often described as a bearish market. }
 Having estimated the volatility $\sigma^*$, we simulate the associated GBM prices with jumps, see \cite{etheridge2002course}, using Monte Carlo simulations. That is, for~$t \in [0, T]$, we generate the price governed by the following stochastic differential equation:
\begin{align} \label{eq: GBM with jump: stock dynamics}
     % S_t &= S_0 \exp \left( \left(\mu - \frac{1}{2} \sigma^2\right) t + \sigma W_t\right)\\
    S_t &= S_0 \exp \left( \left(\mu^* - \frac{1}{2} {\sigma^*}^2\right) t + \sigma^* W_t\right)(1 - \delta)^{N_t},
\end{align}
 where $W_t := \{W(t): t\geq 0\}$ is a standard Wiener process, $\mu^*$ is the \textit{drift} constant, $\sigma^*$ is the volatility constant, $N_t = \{N(t): t\geq 0\}$ is a Poisson process with $P(N_t = k) = \frac{ (\lambda t)^k}{k!} e^{-\lambda t }$ that is independent with $W_t$, $\lambda$ is the average rate of the jump that occurs for the process, and $\delta \in [0, 1)$ is the magnitude of the random jump.\footnote{For 252 daily data, the drift rate and volatility constants can be approximated by using  $\mu^* \approx 252 \mu$ and $\sigma^* \approx \sqrt{252}\sigma$. When $\delta = 0$, Equation~(\ref{eq: GBM with jump: stock dynamics}) reduces to GBM. While it is not shown in this paper, the double linear policy~(\ref{eq: double linear policy}) assures RPE for the GBM case as well.}

 To simulate the price, we discretize the process~(\ref{eq: GBM with jump: stock dynamics}) by taking a time period length of $\Delta t := 1/252$ and $T = 1$ for one year with an annualized drift rate~$\mu^* \in (-1, 1)$, annualized volatility computed from historical data $\sigma_{\rm AAPL} \approx 35.63 \% $, jump intensity $\lambda = 0.2$ with a jump size $\delta = 0.1$. With initial account value $V_0 = 1$, 
 we consider four admissible weighting functions defined by~$w_i:\{0, 1, \dots, N=252\} \to \mathcal{W} \subseteq \mathbb{R}$ for~$i \in \{0,1,2,3\}$ with
 \begin{align*}
     w_0(k) & := 0.8;\\
     w_1(k) &:= \log\left( 1+\frac{k}{N}(e-1) \right); \\
     w_2(k) &:= \frac{1}{2} \left( \sin \left(\frac{1}{\frac{0.02}{N}k - 0.01} \right) + 1 \right); \\
     w_3(k) &:= 
           f(k)\sin \left( \frac{1}{f(k)} \right) \mathbb{1}_{ \left\{  f(k)\sin \left( \frac{1}{f(k)} \right)  \geq 0 \right\}  }(k),
 \end{align*}
 where $ f(k) := (\frac{4}{N}k - 2)$ and $\mathbb{1}_A(x)$ is an indicator function satisfying $\mathbb{1}_A(x) = 1$ for $x \in A$ and zero otherwise. 
 
 The four weighting functions above represent different investment philosophies. For example, $w_0(k)$ represents a constant buy-and-hold strategy, $w_1(k)$ represents an increasing investing strategy over the specified period,  $w_2(k)$ corresponds to a more active trading approach, and $w_3(k)$ represents investing more at the beginning and end of the period, with little or no investment in the middle.
 Consistent with the simulations conducted in~\cite{hsieh2022robust}, we generate $10,000$ GBM sample paths for each $\mu^*$ and various~$\alpha \in \{0.1, 0.3, 0.5, 0.7, 0.9\}$. Then we calculate the average cumulative gain-loss; see Figure~\ref{fig: Gain-Loss Function Simulation with jump}. 
For~$\alpha = 0.5$, the positive expectation gain is seen for all four weighting functions.

 % \begin{figure}[h!]
 %     \centering
 %    \includegraphics[width=\linewidth]{figs/GainLoss_4weight_GBM.eps}
 %     \caption{Gain-Loss Function versus $\mu \in (-1, 1)$ simulate by GBM}
 %     \label{fig: Gain-Loss Function Simulation}
 % \end{figure}

 \begin{figure}[h!]
     \centering
    \includegraphics[width=.9\linewidth]{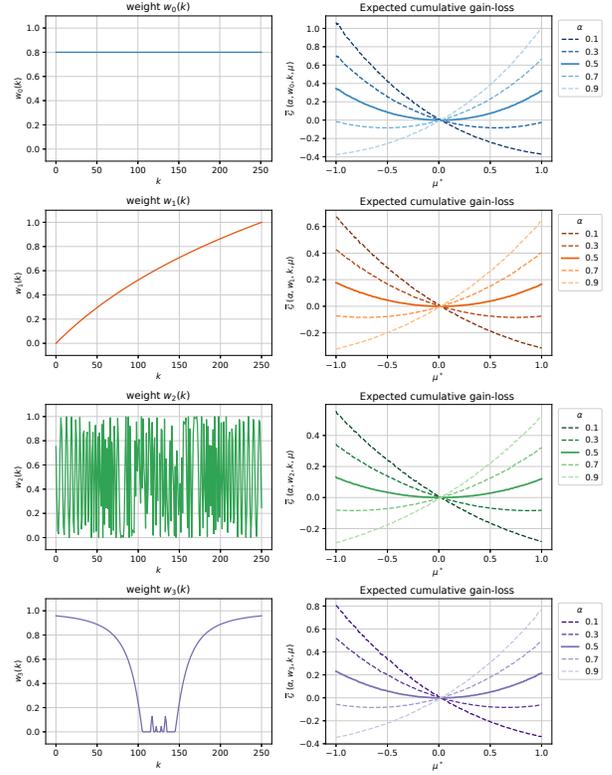}
     \caption{Weighting Functions (Left) and  Expected Gain-Loss for~$\mu^* \in (-1, 1)$.}
     \label{fig: Gain-Loss Function Simulation with jump}
 \end{figure}

 % \footnotetext{$\mu = \mu^* + \lambda \delta$, where $\mu^*$ is the drift constant for GBM, and $\mu \approx \mu^*$ when jump compensator $\lambda \delta$ is too little to ignore.}
 
\end{example}

\begin{example}[Minute-by-Minute Case] \rm
\label{example: minute-by-minute}
In this example, we study the performance of the double linear policy using relatively high-frequency minute-by-minute price data for Twitter Inc. (Ticker: TWTR) between May 4, 2022, and May~19,~2022.\footnote{During this period, CEO Elon Musk announced that the Twitter deal was temporarily put on hold on May 13, causing a 9.7\% decreases in shares at market close. The data is retrieved using the Bloomberg Terminal.}
The price trajectory for the specific period is shown in Figure~\ref{fig: minutely stock price}.
The figure also includes a subplot with a magnified view for the interval~$k \in [50, 100]$ minutes, featuring various moving average lines, which will be used in the next example.

\begin{figure}[h!]
    \centering    \includegraphics[width=.9\linewidth]{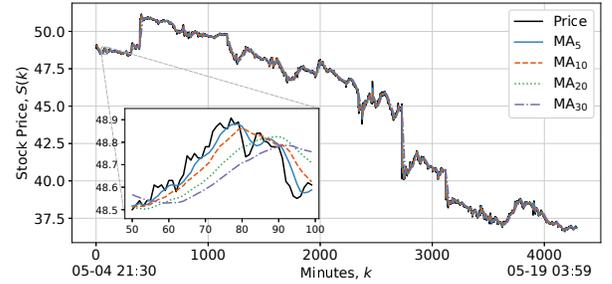}
    \caption{Twitter Minutely Prices from May 4, 2022 to May 19, 2022.}
    \label{fig: minutely stock price}
\end{figure}

We now examine the trading performance of the double linear policy using the same four weighting function $w_i(k)$ for $i \in \{0,1,2,3\}$ described in Example~\ref{example: GBM with jumps}. Specifically, with~$\alpha = 1/2$ and initial account value~$V_0 = 1$, the corresponding trading gain-loss trajectories are shown in Figure~\ref{fig: TWTR trading gain loss}. 
In contrast to the negative returns obtained by the buy-and-hold  (B\&H)  long-only strategy with constant weight $w_0$, we note that all the proposed weighting functions  of the double linear policy assured positive trading gains for the Twitter data.
Table~\ref{tab: minutely data trading performance} also summarizes another performance metric, such as variances and Sharpe ratio. 
It is also worth mentioning that similar findings hold for flipped TWTR price data, indicating the robustness of double linear policy and an ability to capture underlying market dynamics in both bull and bearish markets.

\begin{figure}[h!]
    \centering
    \includegraphics[width=.9\linewidth]{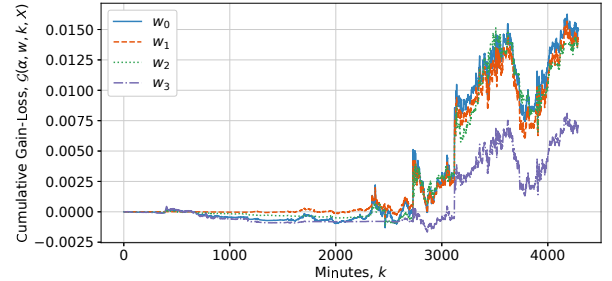}
    \caption{Cumulative Gain-Loss Using Twitter Minute-by-Minute Data.}
    \label{fig: TWTR trading gain loss}
\end{figure}

\begin{table}[h!]
    \centering
    \caption{Performance of Double Linear Policy with Various Weights}
    \begin{tabular}{c c c c c c}
         \hline
         & B\&H & $w_0$ & $w_1$ & $w_2$ & $w_3$ \\
        \hline
        \hline
       {\rm Gain-Loss} & -0.2003 & 0.0150 & 0.0142 & 0.0138 & 0.0070 \\
       \hline
       \makecell[c]{Variance} & 0.0066 & 2.9e-05 & 2.3e-05 & 2.6e-05 & 6.2e-06 \\
       \hline
        Sharpe Ratio & -1.6878 & 1.2106 & 1.2308 & 1.4722 & 0.9497 \\
        \hline
    \end{tabular}
    \label{tab: minutely data trading performance}
\end{table}
\end{example}

\begin{example}[Blending Moving Average Indicator] \rm
In this example, we blend the use of the \textit{moving average} indicator,  a common method in technical analysis,  as a criterion for designing the weighting function into the double linear policy. This  approach enables dynamic adjustment of the investment based on the indicator. 
The weighting function used in the double linear policy~(\ref{eq: double linear policy}) is defined as
\begin{align} \label{eq: weight with MA indicator}
 w_{{\rm MA}_d}(k) := w \cdot \mathbb{1}_{ \left\{ S(k) >{\rm MA}_d(k) \right\} }(k),
\end{align}
where $w \in \mathcal{W}$ and ${\rm MA}_d(k)$
represent the last $d$-period average stock price for~$d \geq 1$.
For example, in the case of minutely data,~${\rm MA}_5(k)$, ${\rm MA}_{10}(k)$, ${\rm MA}_{20}(k)$, ${\rm MA}_{30}(k)$ represent the last 5-minute, 10-minute, 20-minute, and 30-minute average stock price, respectively.
The investment philosophy is to invest only when the stock price is higher than the  moving average, which signals a buying opportunity.

With $\alpha = 1/2$, $V_0 = 1$, and $w=0.8$, we summarize the cumulative gain-loss, variance, and Sharpe ratio in Table~\ref{tab: minutely data trading performance for moving average}, and the trading trajectories are shown in Figure~\ref{fig: TWTR trading gain loss for moving average}.
From the table, we see that $w_{\rm MA_{20}}$ leads to the best performance in terms of the Sharpe ratio.
In all cases, we see positive returns using the weighting functions incorporated with the moving average indicator. Also, while not demonstrated in this paper, the MA indicator in Equation~(\ref{eq: weight with MA indicator}) can be readily replaced by another technical analysis indicator, such as weighted moving average, moving median, moving average convergence and divergence~(MACD) and so on.

\begin{table}[h!]
    \centering
    \caption{Performance of Double Linear Policy with MA Indicator}
    \begin{tabular}{c c c c c}
         \hline
         & $w_{{\rm MA}_5}$ & $w_{{\rm MA}_{10}}$ & $w_{{\rm MA}_{20}}$ & $w_{{\rm MA}_{30}}$\\
        \hline
        \hline
       {\rm Gain-Loss} & 0.0061 & 0.0076 & 0.0117 & 0.0042\\
       \hline
       Variance & 9.7e-06 & 1.1e-05 & 1.4e-05 & 2.2e-06\\
       \hline
        Sharpe Ratio & 0.7097 & 0.8806 & 1.6343 & 0.8651\\
        \hline
    \end{tabular}
    \label{tab: minutely data trading performance for moving average}
\end{table}

\begin{figure}[h!]
    \centering    
    \includegraphics[width=.9\linewidth]{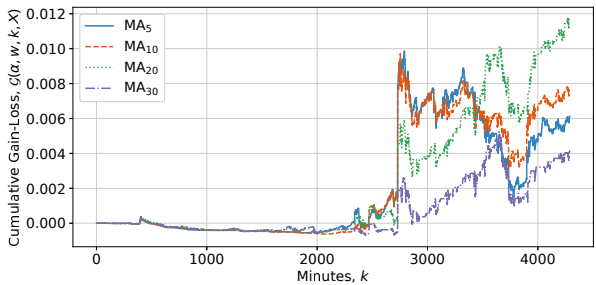}
    \caption{Cumulative Gain-Loss Using Various MA Indicators.}
    \label{fig: TWTR trading gain loss for moving average}
\end{figure}

 \end{example}

\section{Concluding Remarks}
This paper extends  the double linear policy by incorporating time-varying weights in a discrete-time setting. Using a set of elementary symmetric polynomials, we prove that the RPE property is preserved in the extended policy. In addition, we derive an explicit expression for the expected cumulative gain-loss function and its variance.
We conducted extensive Monte Carlo simulations using various weighting functions to validate our theory. Our results also show that the extended double linear policy with time-varying weights can be integrated with the standard technical analysis technique such as moving~average.

In future research, it would be interesting to expand our analysis to a multi-asset case, where the weights can be optimized for a portfolio of assets; see  \cite{hsieh2022robust} for an initial approach. Additionally, one valuable direction would be to investigate the impact of serial-correlated returns on the performance of the double linear policy with time-varying weights.
For example, an Auto-Regressive (AR) return model might be worth pursuing.
Finally, the impact of transaction costs could be considered to assess the practicality of the proposed policy in real-world applications; see \cite{hsieh2022robustness}.

\bibliographystyle{ieeetr}
\bibliography{refs}

\end{document}